\documentclass[12pt,reqno]{amsart}

\usepackage{setspace}

\usepackage{xcolor}
\usepackage[utf8]{inputenc}
\usepackage{graphicx}
\usepackage{amssymb}
\usepackage{amsmath}
\usepackage{epstopdf}
\usepackage{amsthm}
\usepackage{mathrsfs}
\usepackage{enumerate}
\usepackage{listings}
\usepackage{tikz}

\newtheorem{definition}{Definition} 
\numberwithin{definition}{section}
 
\newtheorem{proposition}[definition]{Proposition} 
\newtheorem{theorem}[definition]{Theorem}

\theoremstyle{definition}
\newtheorem{example}[definition]{Example}

\newcommand{\M}{\mathcal{M}}
\newcommand{\A}{\mathcal{A}}
\newcommand{\U}{\mathcal{U}}

\renewcommand{\min}{\operatorname{min}}
\renewcommand{\max}{\operatorname{max}}

\newcommand{\rk}{\operatorname{rk}}

\newcommand{\ncM}{\mathcal{M}}

\newcommand{\poly}{\chi} 

\begin{document}
\title{The Arboricity Polynomial}
\author{Felix Breuer}
\address{Research Institute for Symbolic Computation \\ Johannes Kepler University Linz}
\email{felix@fbreuer.de}
\author{Caroline J.\ Klivans}
\address{Departments of Applied Mathematics and Computer Science \\Brown University}
\email{klivans@brown.edu}

\begin{abstract} We introduce a new matroid (graph) invariant, the arboricity polynomial.  Given a matroid, the arboricity polynomial enumerates the number of covers of the ground set by disjoint independent sets. 
 We establish the polynomiality of the counting function as a special case of a scheduling polynomial, i.e. both in terms of  quasisymmetric functions and  via Ehrhart theory of the normal fan of the matroid base polytope.
 We show basic properties of the polynomial and demonstrate that it is not a Tutte invariant. Namely, the arboricity polynomial does not satisfy a contraction / deletion recursion. 

\end{abstract}

\maketitle

\section{Introduction}
\label{intro}

The arboricity of a graph is a numerical invariant first introduced by
Nash-Williams and Tutte~\cite{Nash, Tutte}.  Given a simple graph $G = (V,E)$,
the arboricity of $G$, $a(G)$, is defined to be the minimum number of forests
needed to decompose (equivalently cover) the edges of $G$:
$$a(G) = \min_k |\{F_1, \ldots, F_k \, | \, F_i \textrm{ is a forest}, \bigcup F_i = E \} |.$$
The definition of arboricity was extended to matroids by
Edmonds~\cite{Edmonds}.  For a matroid $M$ on ground set $E$, the arboricity of
$M$, $a(M)$, is the minimum number of independent sets needed to
decompose (equivalently cover) the ground set of $M$:

$$a(M) = \min_k |\{I_1, \ldots, I_k \, | \, I_i \textrm{ is an independent set}, \bigcup I_i = E \} |.$$

If $M(G)$ is the graphical matroid associated to a graph $G$, then the arboricity of $M$ recovers the arboricity of $G$, $$a(M(G)) = a(G).$$

The constructive problem associated to arboricity is to explicitly find
a partitioning of a matroid into as few independent subsets as
possible and is known as the matroid partitioning problem, see e.g.
\cite{edmonds2010matroid, edmonds1965transversals, knuth1973matroid}.

Here we expand the consideration of arboricity to the enumerative question 
 of how many
partitions there are of a fixed size.  Specifically, given a matroid $M$ on a ground set
$E$ and a positive integer $k$, define 
$$\mathcal{A}_M(k) = \textrm{\# partitions of $E$ into independent sets with at most $k$
parts.}$$  Our main theorem shows that the counting function
$\mathcal{A}_M(k)$ is a polynomial function in $k$ with integer coefficients, which we call the
{\bf{arboricity polynomial}}.

Interestingly, unlike many polynomials associated to a matroid, the arboricity polynomial is not a Tutte evaluation.  Namely, the arboricity polynomial does {\emph{not}}
generally satisfy a contraction/deletion recursion.  Hence we can not use a standard induction argument to show it is a polynomial.  
Instead, we identify the arboricity polynomial as a special case of a
scheduling polynomial \cite{Scheduling}.  From this perspective, the
arboricity polynomial can be seen either as a polynomial
specialization of a quasisymmetric function or as the union of Ehrhart
polynomials arising from a triangulation of a polytope.  The theory of scheduling
 allows us to associate a natural geometry to arboricity.
More specifically, matroid partitions can be seen as integer points contained in
polyhedral cones defined by the matroid base polytope.

In Section~\ref{arboricty}, we define arboricity, the arboricity polynomial and consider examples.  In Section~\ref{sec:schedule} we review the theory of scheduling polynomials.  Section~\ref{sec:main} establishes arboricity as an instance of scheduling and as such contains the main theorem, that the scheduling counting function is a polynomial.  Section~\ref{sec:main} also demonstrates that the arboricity polynomial is neither a contraction deletion invariant nor valuative with respect to the matroid base polynomial.  

\subsection{Acknowledgements}  The authors thanks Teressa Chambers and Christopher Eur for valuable conversations.  The authors especially thank Christopher Eur for the question of valuativity and working out the example of Section~\ref{value} with the second author during a visit to Carnegie Mellon University.

\section{Arboricity}
\label{arboricty}

Let $[m]$ denote the set $\{1, 2, \ldots, m\}$.
Let $G = (V,E)$ be a finite graph.  An \emph{independent cover} of $G$ is a mapping 
$$g : E \rightarrow [m] \, \, \textrm{such that} \, \, \forall t \in 
[m], \,  g^{-1}(t) \, \, \textrm{is acyclic.}$$ 
 Namely, $g$ partitions the edges of $G$ into a disjoint collection of forests.  Similarly, for any finite matroid $M$ on ground set $E$, an \emph{independent cover} of $M$ is a mapping 
 $$g : E \rightarrow [m] \, \,  \textrm{such that} \, \,  \forall t \in [m],  \, g^{-1}(t) \, \,  \textrm{is independent.}$$

Note that if  $m \geq |E|$, then an independent cover always exists; map each edge (element of the ground set) to a distinct value.  

\begin{definition} The \emph{arboricity} of a graph $G$, $a(G)$, is the minimum value of $m$ necessary for the existence of an independent cover $c : E(G) \rightarrow [m]$.
\end{definition}

\begin{definition} The \emph{arboricity} of a matroid $M$, $a(M)$,  is the minimum value of $m$ neccessary for the existence of an independent cover $c : E(M) \rightarrow [m]$. 
\end{definition}

Graphical arboricity is a special case of matroid arboricity.  
If $M(G)$ is the graphical matroid associated to $G$, then the independent sets of $M$ are precisely the forests of $G$, hence $a(M(G)) = a(G)$.

An appropriate parallel concept to keep in mind is the chromatic number of a graph.  A proper coloring of $G$ is a mapping from the \emph{vertices} of $G$ to a finite set $[m]$ such that the inverse image of any value is an independent set of vertices, i.e. a collection of vertices no two of which span an edge.  The chromatic number is the minimum value of $m$ necessary for the existence of a proper coloring.   In fact, the theory of arboricity is often referred to in the literature as matroid coloring.  \\

\begin{figure}
    \centering
    \includegraphics[width=0.5\linewidth]{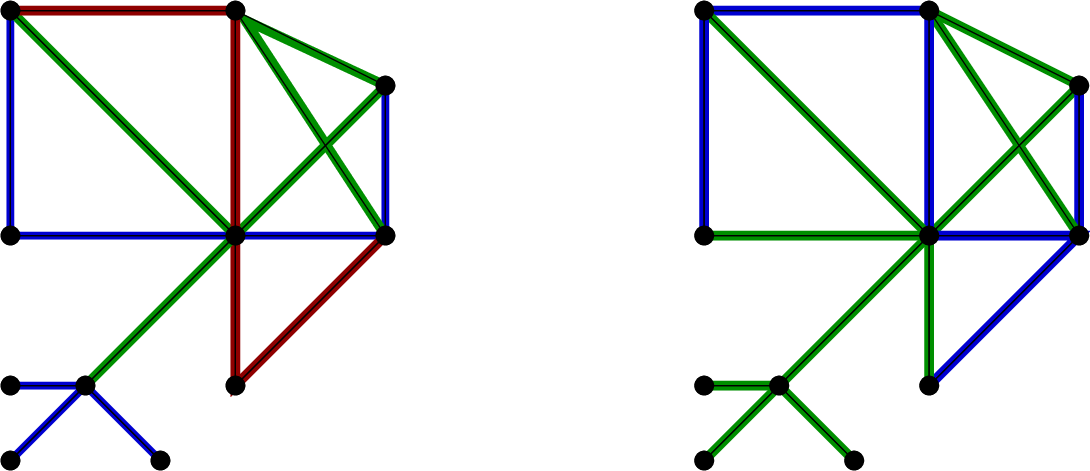}
    \caption{On the left, a graph decomposed into $3$ forests.  On the right, the same graph decomposed into $2$ forests.}
    \label{fig:arborexample}
\end{figure}

\begin{example}
In Figure~\ref{fig:arborexample}, two different decompositions of the same graph $G$ are shown.  On the left, the graph is decomposed into $3$ forests, thus demonstrating that the arboricity is at most three, $a(G) \leq 3$.  On the right, $G$ is decomposed into $2$ forests, thus demonstrating that $a(G) \leq 2$.  Since $G$ contains cycles, the arboricity is also at least two, therefore $a(G) = 2$ for this example.   \end{example}

Early work on arboricity focused on the value $a(M)$ as well as the algorithmic problem of finding a partition of size $a(M)$.    
As shown in 
~\cite{Nash, Tutte} and extended in~\cite{Edmonds}, the arboricity $a(M)$ is given explicitly 
by the formula:
$$ a(M) = \max\limits_{{X\subseteq E}} \bigg{\lceil} { \frac{|X|}{\rk(X)}} \bigg{\rceil}, $$
where $X$ ranges over all possible subsets of the ground set and $\rk(X)$ is the matroidal rank of $X$.\\

Much work has also been done investigating specializations of arboricity.
For example, the papers~\cite{algor1989star, aoki1990star, gonccalves2009star} consider \emph{star arboricty} which is
the minimum number of star forests needed to cover the edges of a
graph.  Similarly, the works~\cite{Hararylinear, alon1988linear, Foxlinear} consider \emph{linear arboricity}
which is the minimum number of paths needed to cover the edges of a
graph.  The longstanding linear arboricity conjecture has generated novel recent work connecting arboricity to various graph theoretic methods.  It postulates that 
$$\textrm{linear } a(G) \leq \ \left\lceil\frac{\Delta}{2}\right\rceil$$
where $\Delta$ is the maximal degree of $G$~\cite{la,Foxlinear}
  Specialization also occurs at the structural level of the graph;
i.e. considering arboricity for planar graphs or degree regular
graphs, see for example~\cite{Wuplanar, Guldanregular,  Enomotoregular}.\\

Returning to coloring, the primary parallel we will draw is to the chromatic polynomial of a graph $\chi_G(k)$.  The chromatic polynomial  is the counting function which records the number of proper colorings of $G$ using at most $k$ colors.

The arboricity polynomial is defined similiarly:

\begin{definition}
For a matroid $M$, on ground set $E$,  the \emph{arboricity polynomial} $\mathcal{A}_M(k)$ is the counting function which records the number of independent covers $g : E \rightarrow [k]$.
  \end{definition}
Hence $\mathcal{A}_G(k)$ records the number of labeled partitions of the edge set of $G$ into independent sets with at most $k$ parts; just as $\chi_G(t)$ records the number of labeled partitions of the vertex set of $G$ into independent sets with at most $t$ parts.  The arboricity of a graph $G$ with edge set $E$ is that of its graphical matroid or equivalently the counting function which records the number of independent covers $g: E \rightarrow [k]$.

We will justify calling the arboricity function a polynomial in the next section.  A typical way to show that the chromatic function $\chi_G(t)$ is a polynomial is via a contraction--deletion recursion and induction.  We require different techniques as arboricity does not in general satisfy a contraction--deletion recursion. \\

Although quite natural, the arboricity counting function has been overlooked until now perhaps precisely because it is not a deletion--contraction invariant.  Arboricity is however particularly natural in the geometric view of matroids and  matroid polytopes, a viewpoint of strong recent importance.  \\

Let us consider a few examples before moving to the general theory.  
\begin{example}
\label{ex:allfree}
Let $M$ be the free matroid on a ground set of size $n$, i.e. all subsets of the ground set are independent.  Then $\mathcal{A}_M(k) = k^n.$  Equivalently, let $G$ be an acyclic graph on $n$ edges (whether it is a tree or forest is not relevant).  Then any partition of the edges yields an independent cover, the graphical matroid $M(G)$ is a free matroid, and $$\mathcal{A}_{M(G)}(k) = k^n.$$    

\end{example}

\begin{example}
\label{ex:c3}
    Consider the graph which is the three cycle $C_3$.  The graphical matroid  $\M(C_3$) has a ground set of size three, and all subsets of size one or two are independent.  We will derive the general form of $\A_{\M(C_n)}(k)$ in Proposition~\ref{prop:cycles}.  For now, direct calculation gives that 
    $$ \A_{\M(C_3)}(k) = k^3 - k.$$
    
\pagebreak   
     
        Therefore 
    \begin{align*}
    \mathcal{A}_{M(C_3)}(1) &= 0 \\ 
    \mathcal{A}_{M(C_3)}(2) &= 3*2 = 6 \\
    \mathcal{A}_{M(C_3)}(3) &= 27-3 = 24. \\
\textrm{Deleting any edge $e$ gives a path of length} & \textrm{ $2$}  \textrm{ which is acyclic.  Therefore \color{white}{{ABCD}}}\\
\mathcal{A}_{M(C_3 - e)}(1) &= 1 \\ 
\mathcal{A}_{M(C_3 - e)}(2) &= 2^2 = 4 \\
\mathcal{A}_{M(C_3 - e)}(3) &= 3^2 = 9. \\
\textrm{Contracting any edge $e$ gives a double } &  \textrm {edge which} \textrm{ has no independent } \\
\textrm{ covers of size $1$, two of size $2$ and six } & \textrm{of size $3$:}   \\
\mathcal{A}_{M(C_3 / e)}(1) &= 0 \\
\mathcal{A}_{M(C_3 / e)}(2) &= 2 \\
\mathcal{A}_{M(C_3 / e)}(3) &= 6.
\end{align*}\\

 Therefore, already for the cycle of length $3$,  we see that the arboricity polynomial may not satisfy the contraction / deletion recursion:
$$\A_{\M(C_3)}(k)  \neq \mathcal{A}_{M(C_3 - e)}(k) \pm \mathcal{A}_{M(C_3 / e)}(k), $$
see also Section~\ref{sec:condel}.

\end{example}

Next we give an example of a matroid that is not isomorphic to a graphical matroid for any graph.

\begin{example}

Let $\M = \mathcal{U}_{r,n}$ be the uniform matroid of rank $r$ on a ground set of size $n$.  In general, $\M$ is a realizable but not graphical matroid. Again direct calculation gives that

\begin{align*}
    \A_{\mathcal{U}_{2,3}}(k) &= k^3 - k \\
    \A_{\mathcal{U}_{2,4}}(k) &= k^4 - 4k^2 + 3k \\
    \A_{\mathcal{U}_{2,5}}(k) &= k^5 - 10k^3 + 15k^2 - 6k.
    \end{align*}

The matroid $\mathcal{U}_{2,3}$ is isomorphic to  the graphical matroid $\M(C_3)$, hence the agreement of arboricity polynomials. The matroids  $\mathcal{U}_{2,4}$ and $\mathcal{U}_{2,5}$ are not graphical.

\end{example}
General properties of the arboricity polynomial are discussed in Section~\ref{sec:main}.

\section{Scheduling}
\label{sec:schedule}

Arboricity can be seen as a special case of a more general theory of \emph{scheduling} which we review in this section.  
In~\cite{Scheduling}, the authors considered {scheduling problems} of the
following form.

\begin{definition}A \emph{scheduling problem} on $n$ items is a Boolean formula
$S$ over atomic formulas $x_i \leq x_j$ for $i,j \in [n]$.  
A \emph{$k$-schedule} solving $S$ is an assignment  $\omega
:[n] \rightarrow [k]$ such that $S(\omega)$ is true. 
\end{definition}
The $n$ items
are interpreted as tasks that need to be completed.  The atomic
formulas specify constraints on the order in which the tasks can be completed.

\begin{example}
Consider the scheduling problem $S$ on $5$ items:
$$ S = (x_3 \leq x_4) \land (x_5 \leq x_1) \land (x_4 \leq x_3). $$
In words, the constraints dictate that tasks $3$ and $4$ need to occur at the same time and that task $5$ can not occur after task $1$.

If $k = 3$, i.e. there are $3$ possible time slots, then the $k$-schedules include:
$$(1,3,2,2,1), (3,1,3,3,2), (1,1,1,1,1), (2,1,2,2,1)$$
and do not include:
$$(2,3,2,2,1), (3,2,2,3,2), (1,2,3,2,3), (1,1,2,1,1). $$

\end{example}

It is possible that $k > n$, the number of time slots might be larger than the number of items.  This makes the scheduling problem easier to satisfy.

In the example above, if $k = 8$ then the $k$-schedules satisfying $S$ also include:
$$(8,2,6,6,8), (7,1,1,1,5), (7,1,1,1,6), (4,8,6,6,3). $$

We interpret $k$-schedules (solutions to a scheduling problem) as integer points in $\mathbb{Z}^n$. 
Importantly, note that whether or not an integer point ${\bf{x}}$ is a solution
to a scheduling problem $S$ only depends on the relative ordering of
the coordinates of ${\bf x}$.

For example, if the point $(4,5,5,2)$ satisfies a
scheduling problem $S$ then so does $(2,3,3,1)$ and all integer points
$(x_1, x_2, x_3, x_4)$ such that $x_4 < x_1 < x_2 = x_3$.   The value $k$  bounds the entries componentwise.

Considering the relative ordering of the coordinates allows us to associate solutions to scheduling problems combinatorially as collections of ordered set partitions and geometrically in terms of the Braid arrangement, as explained next.

\begin{definition}
    The \emph{Braid arrangement} $\mathcal{B}_n \in \mathbb{R}^n$ is the hyperplane arrangement consisting of all hyperplanes of the form 
    $ x_i = x_j $  for all $i,j \in [n].$
\end{definition}

For any integer point $\bf{x} \in \mathbb{R}^n$ its relative position to the Braid arrangement consists of the information of whether $\bf{x}$:

\begin{itemize}
    \item  
lies directly on a hyperplane: ($x_i = x_j$ or $x_i \neq x_j$) 
    \item or 
 which side of the hyperplane it lies on ($x_i < x_j$ or $x_i > x_j$).  
\end{itemize}

Any cone of the Braid arrangement $\mathcal{B}_n$ represents the
collection of points in $\mathbb{R}^n$ whose relative positions with respect to
all hyperplanes of $\mathcal{B}_n$ are the same. For example,
continuing from above, all $\bf{x} \in \mathbb{R}^4$ s.t.  $$x_4 < x_1
< x_2 = x_3$$ form a three dimensional cone of the Braid arrangement
$\mathcal{B}_4$ which we represent by the ordered set
partition: $$4|1|23.$$ This three dimensional cone is contained in the
closure of two full dimensional regions of $\mathcal{B}_4$, the four
dimensional cone of points s.t.  $$x_4 < x_1<x_2<x_3 \,
\longleftrightarrow \, 4|1|2|3$$ and the cone $$x_4<x_1<x_3<x_2 \,
\longleftrightarrow \, 4|1|3|2.$$
See~\cite{Scheduling} for more discussion on the association of ordered set partitions and the Braid arrangement.  \\

Next consider the parameter $k$ appearing in $k$-schedules.  
The points in $\mathbb{R}^n$ all of whose coordinates are bounded by a fixed value $k$.  Geometrically such points form an $n$-dimensional cube of side length $k$.  

Bringing these two ideas together, {\em the solutions to a satisfiable scheduling problem can be seen as the integer points in a union of Braid cones intersected with a cube,} see Figure~\ref{fig:schedule}.

\begin{figure}
    \centering
\includegraphics[width=0.8\linewidth]{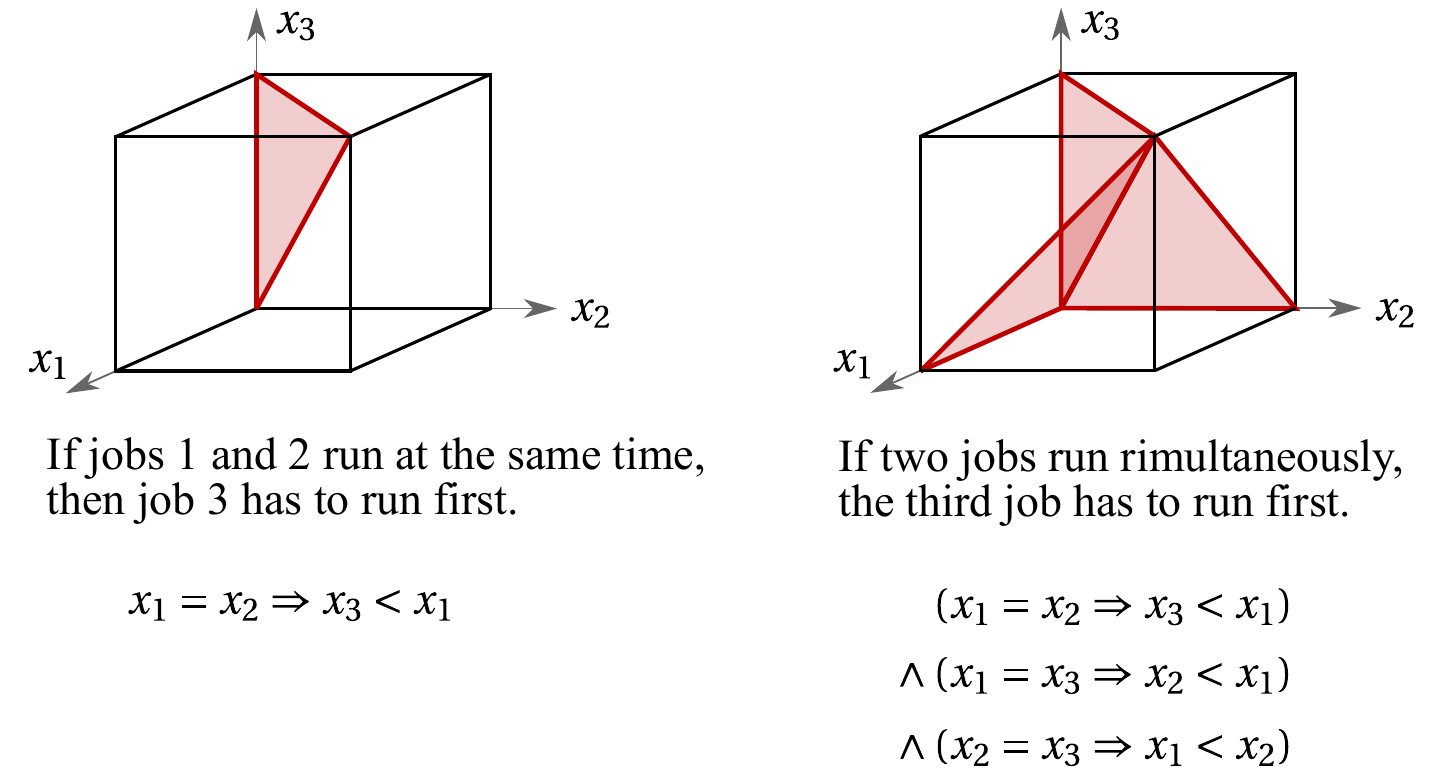}
    \caption{Scheduling Geometry.  Suppose the cubes shown have side length $k$.  The $k$-schedules of the described scheduling problems are the integer points inside the cube that do not lie on any of the shaded regions. }
    \label{fig:schedule}
\end{figure}

\begin{example}
    
Classical graph (vertex) coloring can naturally be seen in the framework of scheduling.  A proper  $t$-coloring of a graph $G$ is a map: $$ c(G): V(G) \rightarrow [t], \, \, \, \,   c(v_i) \neq c(v_j), \, \, \, ij \in E(G)$$
The corresponding scheduling problem is a conjunction over all edges of the graph.  $$S = \bigwedge_{ij \in E(G)} x_i \neq x_j$$

Geometrically, solutions (colorings) 
are integer points 
not on the collection of planes $\{(x_i =x_j), \, \, \, ij \in E(G)\}$.   Given a graph $G$, the collection of hyperplanes corresponding to the edges of $G$ is known as the graphical hyperplane arrangement corresponding to $G$ and is always a subarrangement of the Braid arrangement.  The Braid arrangement $\mathcal{B}_n$ can be thought of as the graphical hyperplane arrangement associated to the complete graph $K_n$.

\begin{figure}[h]
\begin{centering}
\includegraphics[angle=270,width=0.6\linewidth]{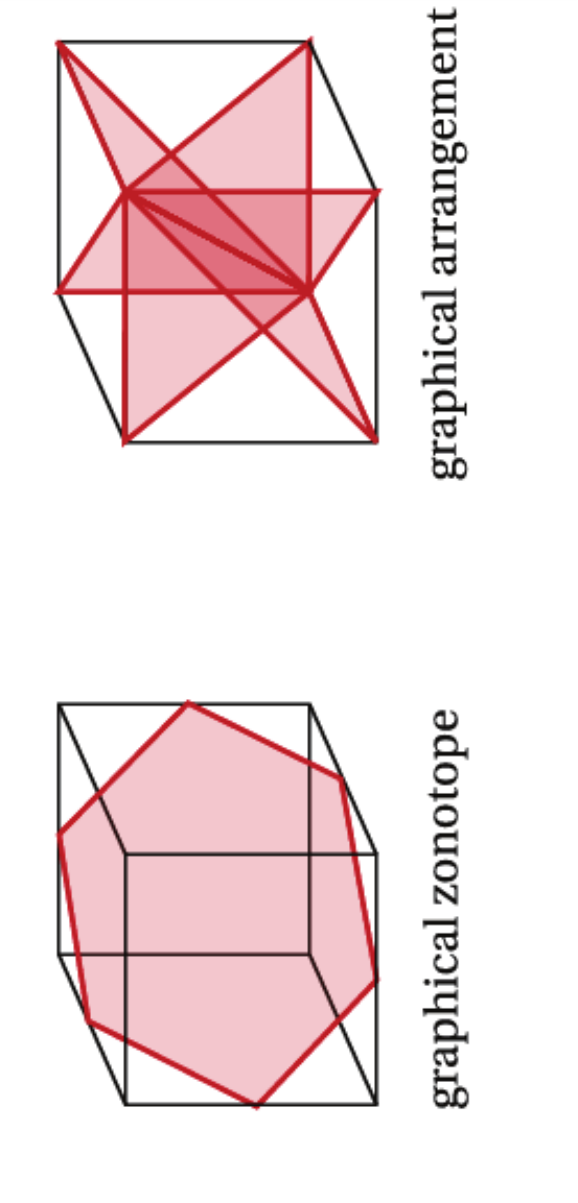}
\caption{The graphical zonotope of the complete graph $K_3$ and the corresponding graphical arrangement.} 
\label{fig:normalfan}
\end{centering}
\end{figure}

Relative interiors of the graphical arrangement are the interiors of the maximal cones of the normal fan of the corresponding graphical zonotope, see Figure~\ref{fig:normalfan} and~\cite{Scheduling}.  Therefore, the geometric perspective of coloring coming from scheduling  can be summarized as: proper colorings of a graph correspond to the non-negative integer points lying in the relative interiors of maximal cones of the normal fan of the corresponding graphical zonotope. 
\end{example}

Given a scheduling problem $S$, counting the number of $k$-schedules is now a question of counting integer points in certain convex regions.   Such enumerative questions are generally quite difficult and have a rich literature, see e.g.~\cite{AB,beck}. 

In,~\cite{Scheduling} the authors use two sets of techniques to understand the
integer points corresponding to $k$-schedules -- Ehrhart theory and
the theory of quasisymmetric functions in non commuting variables,
NCQSym.

Ehrhart theory is directly concerned with counting integer points in
convex regions.  The connection to NCQSym is via ordered set
partitions and is more natural for the arboricity setting.  We will not give a full treatment here and refer the reader to~\cite{bergeron2008hopf}. 

The
algebra NCQSym has a basis known as the monomial basis which is
indexed by ordered set partitions (it specializes to the monomial basis for QSym
which is indexed by compositions).  Let $\{x_1, x_2, \ldots \}$ be a collection of non-commuting
variables.  Given $a \in \mathbb{Z}^n$, let $\Delta(a)$
be the ordered set partition $(\Delta_1 | \Delta_2 | \ldots | \Delta_k)$ such that 
$a$ is constant on each set $\Delta_i$
and satisfies $a|_{\Delta_i} < a|_{\Delta_{i+1}}$ for all $1 \leq i \leq k$.

Let $\Phi$ be an ordered set partition.  The monomial nc-quasisymmetric function indexed by $\Phi$ is:
$$\ncM_{\Phi} := \sum_{a \in \mathbb{N}^n \, \Delta(a) = \Phi} {\bf{x}}_{a},$$
where ${\bf{x}}_{a}$ is the monomial ${\bf{x}}_{a} = x_{a_1}x_{a_2}\ldots x_{a_n}$.   $\ncM_{\Phi}$ therefore encapsulates solutions to a scheduling problem of fixed relative order.   In~\cite{Scheduling}, using this connection and known specializations of quasisymmetric functions, the authors show the following.

\begin{theorem}\cite[Theorem 2.4]{Scheduling}
\label{fbasis}
Given a scheduling problem $S$ on $n$ items, the scheduling counting function,
$\poly_S(k)$
is a polynomial in $k$ of degree at most $n$, the \emph{scheduling polynomial} of $S$. 

\begin{equation}
\label{eqn:fbasis}
  \poly_S(k) = \sum_{i=1}^n f_i\binom{k}{i} = \sum \M_{\Phi}(\bf{1}^k),
\end{equation}
 where the coefficients $f_1,\ldots,f_n$ 
are non-negative integers counting the number of ordered set partitions $\Phi$ with $i$ non-empty blocks such that $S(\Phi)$ holds.
\end{theorem}

\section{Arboricity as Scheduling}
\label{sec:main}

In this section we reframe arboricity as a scheduling problem.  
Given a matroid $M$ on a finite ground set
$E$, an independent cover is an ordered set partition of $E$ such
that no block contains a circuit of $M$. Let $\mathcal{C}$ denote the set of circuits of $M$, then the arboricity scheduling boolean function is
\begin{equation}
  \label{eq:Arbor}
  A
  _M = \bigwedge_{C \in \mathcal{C}} \neg(x_{i_1} = x_{i_2} = \cdots = x_{i_m}),
  \end{equation}
where the conjunction runs over all circuits $C \in \mathcal{C}$,  $C = \{i_1, i_2, \ldots, i_m\}$.  In terms of scheduling, the collection of jobs which may start at any
 fixed time can not be dependent.

Applying Theorem~\ref{fbasis}, we have the following result.  

\begin{theorem}\label{thm:apoly} Let $M$ be a matroid on ground set $E$ with $|E|=n$.  The counting function
    ${\mathcal{A}}_{M}(k)$ is a polynomial function in $k$ of degree at most $n$ with integer coefficients, ${\mathcal{A}}_{M}(k) \in \mathbb{Z}[x_1, x_2, \ldots, x_n].$ 
\end{theorem}
\begin{proof}
Theorem~\ref{fbasis} applied to Equation~\ref{eq:Arbor} gives the
polynomiality of ${\mathcal{A}}_{M}(k)$.  The claim of integer
coefficients follows from the symmetry of the arboricity problem.  If
an ordered set partition $\Phi = \Phi_1|\Phi_2| \ldots |\Phi_{\ell}$,  satisfies $A_M$ then any
ordered set partition $\Phi_{\sigma(1)}|\Phi_{\sigma(2)}| \ldots |
\Phi_{\sigma(\ell)}$ resulting from permuting the blocks of $\Phi$ will also
satisfy $A_M$.  Therefore each $f_i$ in Equation~\ref{eqn:fbasis} will be a multiple of $(i!)$.
  \end{proof}

Next we explore some basic properties of the arboricity polynomial. 

\begin{proposition}\label{aprop} Let $\mathcal{M}$ be a matroid on ground set $E$ with $|E|=n$.
    \begin{enumerate}
        \item The constant term of $\mathcal{A}_{{\mathcal{M}}}(k)$ is equal to zero; i.e. $k$ divides $\A_{\M}(k)$.
        \item The sum of the coefficients  of $\mathcal{A}_{\mathcal{M}}(k)$ is equal to $0$ unless $\mathcal{M}$ is the free matroid in which case the sum is equal to $1$. 
        \item $\mathcal{A}_{\mathcal{M}}(k)$ is monic of degree equal to the size of the ground set of $\mathcal{M}$. 
    \end{enumerate}
\end{proposition}
\begin{proof} $\,$\\
  \begin{enumerate}    
        \item There are no independent covers of size zero hence $\mathcal{A}_{\M}(0) = 0$. 
        \item There are no independent covers of size one unless the entire ground set is independent, in which case we have the free matroid of Example~\ref{ex:allfree}.  
        \item Consider the presentation of $\mathcal{A}_{\mathcal{M}}(k)$ as given in Theorem~\ref{fbasis}.  The only term which potentially contributes to the coefficient of $k^n$ is the summand when $i = n$.  There is precisely one \emph{non-ordered} partition of $[n]$ into $n$ non-empty blocks and $n!$ ordered set partitions of $[n]$ into $n$ non-empty blocks.  All such partitions yield an independent cover of $E$. Therefore the leading term is determined by the quantity $n! {k \choose n}$.
    \end{enumerate}
\end{proof}

The forthcoming work \cite{GRWC} explores other general properties of arboricity polynomials.  For example, the arboricity polynomial of a direct sum of matroids is equal to the product of the polynomials of the matroids.

\subsection{Cycle Graphs}

Here we compute the general form of the arboricity polynomial for all cycle graphs. We explore the ideas and properties established above through this family. 

\begin{proposition}\label{prop:cycles} Let $M(C_n)$ be the graphical matroid of the cycle graph on $n$
vertices.
The scheduling polynomial of $M(C_n)$ is
$${\mathcal{A}_{M(C_n)}}(k) = k^n - k. $$

\end{proposition}

\begin{proof}
Let $M(C_n)$ be the graphical matroid of the cycle graph on $n$
vertices.  Independent sets of the cycle matroid correspond to acyclic
subsets of edges (forests) of the graph.  For $k \leq 1$ we find that there are no
independent covers of $M$.  For $k \geq 2$, independent covers
correspond to ordered set partitions of $[n]$ with at least $2$ parts.

For $n=3$,  consider the hyperplanes in the braid arrangement in $\mathbb{R}^3$, the ground set  of the
matroid, or edge set  of the graph.  The solutions of the arboricity scheduling problem consist of ordered set partitions
corresponding to collections of integer points which do not lie on the
intersection: $x_1 = x_2 = x_3$. Hence the corresponding nc-quasisymmetric function encoding the solutions is
$${\bf{{\mathcal A}}}(C_3) = \ncM_{(1|23)} + \ncM_{(2|13)} + \ncM_{(3|12)} + \ncM_{(23|1)} + \ncM_{(13|2)} +  \ncM_{(12|3)} + $$ $$ \ncM_{(1|2|3)} + \ncM_{(1|3|2)}
+ \ncM_{(2|1|3)} + \ncM_{(2|3|1)} + \ncM_{(3|1|2)} + \ncM_{(3|2|1)}. $$

Converting to the form given by Theorem~\ref{fbasis} gives the arboricity polynomial.  
In particular, each ordered set partition of length $\ell$ contributes the expression $k \choose \ell$.\footnote{For the familiar reader, the arboricity polynomial is the result of applying the type map to the elements of NCQSYM and then evaluating the resulting quasisymmetric function at ${\bf 1}^k$.}

$$ \A_{(C_3)} = 3 { k \choose 2} + 3 { k \choose 2} + 6 { k \choose 3} = k^3 - k. $$

In general for $C_n$, because there is only one cycle,  solutions are points ${\bf x}$ such that $\neg (x_1 = x_2 =\cdots =x_n)$.  Equivalently, solutions correspond to all ordered set partitions with at least two parts. Hence, 

$$\A_{(C_n)}(k) = \sum_{m=2}^{n} S(n,m) \cdot k(k-1) \cdots (k-m+1), $$
where $S(n,m)$ is the Stirling number of the second kind.   Simplifying the expression above, using the identity that $\sum_0^n S(n,m) \cdot k(k-1) \cdots (k-m+1) = k^n$, we find that the number of covers of the cycle graph with at most $k$ forests is

$${\A_{(C_n)}}(k) = k^n - k. $$

\end{proof}

\subsection{Non Contraction/Deletion}
\label{sec:condel}

Example~\ref{ex:c3} gives the arboricity polynomial for the cycle graph on $3$ vertices which already shows an instance of the  contraction deletion recursion failing.  Proposition~\ref{prop:cycles} examines the general cycle graph and gives an infinite family of examples of arboricity polynomials that do not satisfy contraction/deletion.   We stress that scheduling polynomials in general may or may not satisfy the recursion -- the chromatic polynomial is also an instance of scheduling.

Consider the graphical matroid of the $4$-cycle, $C_4$.  The number of independent covers of size at most $3$ is $3^4-3 = 78$.  Contracting any edge gives the $3$-cycle.   For the graphical matroid of the $3$-cycle, $C_3$, the number of independent covers of size at most $3$ is $3^3-3 = 24$.    Deleting any edge gives a path of length $3$ which is a free matroid and hence has $3^3 = 27$ covers of size at most $3$. 
\begin{align*}
\A_{(C_4)}(3) &= 78 \\
\A_{(C_4 - e)}(3) &= 24 \\
\A_{(C_4 / e)}(3)  &= 27.
\end{align*}

In general,
\begin{align*}
\A_{(C_n)}(k) &= k^n - k \\
\A_{(C_n - e)}(k) &= k^n \\
\A_{(C_n / e)}(k)  &= k^{n-1} - k.
\end{align*}

For $k \neq 2$, $$(k^n - k) \neq k^n \pm (k^{n-1} - k).$$

\subsection{Non Valuative}
\label{value}

The arboricity polynomial is not generally valuative with respect to the matroid base polytope.  The authors thank Christopher Eur for posing this question and working out the following example with the second author during a visit to Carnegie Mellon University.  

Let $\mathcal{U}_{24}$ be the uniform matroid of rank $2$ on $4$ elements. The matroid base polytope of $\mathcal{U}_{24}$ is the octahedron.  It decomposes into three pieces as in Figure~\ref{fig:valuation}: the top pyramid $\M_1$, middle slice $\M'$ and the bottom pyramid $\M_2$. The two pyramids are matroid base polytopes for isomorphic matroids.  To check if the arboricity polynomial is valuative on the matroid polytope of $\U_{24}$, we consider whether or not

$$\A_{{\U_{24}}}(t) = 2{\A_{\M_1}} (t) -\mathcal{A_{\mathcal{M'}}}(t).$$

In $\U_{24}$, all pairs of elements are independent and all triples of elements are circuits. The ordered set partitions corresponding to independent covers are all ordered set partitions of $[4]$ without blocks of size three or four.

$$\A_{\U_{24}}(t) = t^4-4t^2+3t. $$

$\M'$ has bases $\{13, 14, 23, 24\}$.  It is a graphical matroid for the graph with two sets of double edges identified at a vertex. Equivalently, it is the direct sum of two copies of $\U_{12}$.  The arboricity polynomial is
$$ \A_{\M'}(t) = (t^2 - t)^2 = t^4-2t^3+t^2.$$
Finally, the matroid $\M_1$ has bases $\{12, 13, 14, 23, 24\}$,

$$\A_{\M_1}(t) = t^4 - t^3-2t^2 + 2t.$$

Putting these together, 
$$ 2(t^4 - t^3-2t^2 + 2t) - (t^4-2t^3 + t^2) = t^4 - 5t^2 + 4t \neq t^4 - 4t^2 + 3t.$$

\begin{figure}
\centering
\includegraphics[width=0.7\linewidth]{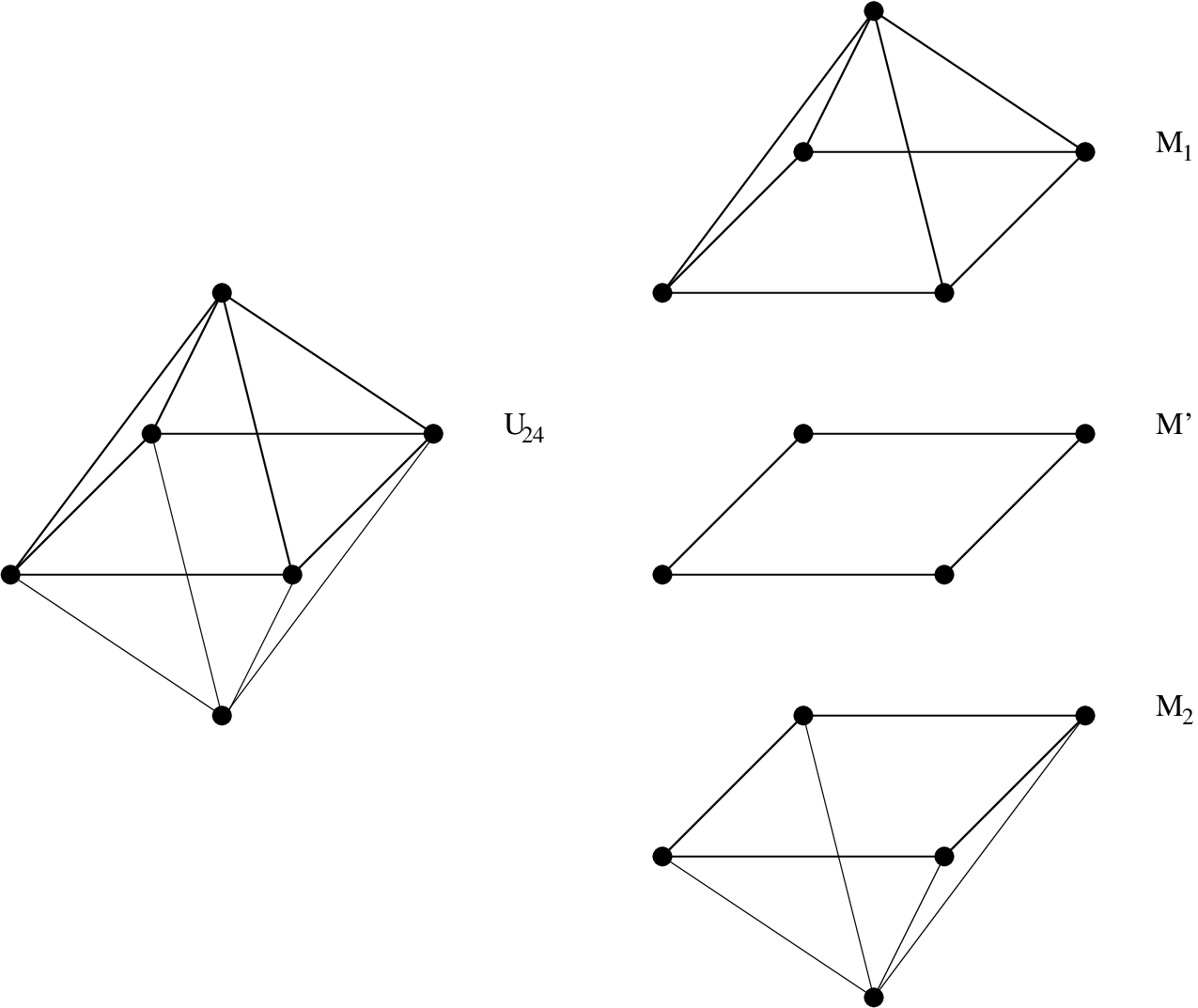}
\caption{Matroid base polytope of $U_{2,4}$}
 \label{fig:valuation}
    \end{figure}

\bibliographystyle{plain}
\bibliography{refs}

\end{document}